\documentclass[graybox]{svmult}
\usepackage{newtxtext}
\usepackage{newtxmath}
\usepackage{graphicx}
\usepackage{makeidx}
\usepackage{multicol}
\usepackage{footmisc}

\usepackage{hyperref}
\hypersetup{colorlinks=true, linkcolor=blue, citecolor=blue}
\usepackage[noabbrev,capitalize,nameinlink]{cleveref}

\crefformat{equation}{(#2#1#3)}
\crefrangeformat{equation}{(#3#1#4) to~(#5#2#6)}
\crefmultiformat{equation}{(#2#1#3)}{ and~(#2#1#3)}{, (#2#1#3)}{ and~(#2#1#3)}
\crefname{assumption}{Assumption}{Assumptions}

\newcommand{\R}{\mathbb{R}}

\newcommand{\Hc}{\mathcal{H}}

\newcommand{\dx}{\,\mathrm{d}x}

\newcommand{\dt}{\,\mathrm{d}t}

\newcommand\restr[2]{\ensuremath{\left.#1\right\vert_{#2}}}

\renewcommand{\D}{\mathcal{D}}

\newtheorem{assumption}{Assumption}

\begin{document}
	\title*{An abstract approach to the Robin--Robin method}
	\author{Emil Engström and Eskil Hansen}
    \institute{Emil Engström, Eskil Hansen \at Lund University, P.O. Box 118, 221 00 Lund, Sweden\\ \email{emil.engstrom@math.lth.se, eskil.hansen@math.lth.se}}
	\maketitle
    \abstract*{Recently, their has been development of an abstract approach to the Robin--Robin method, enabling the treatment of linear and nonlinear elliptic and parabolic equations on Lipschitz domains within one framework. However, previously this setting has not been applicable to initial-boundary value problems. The aim of this short note is therefore to demonstrate that this general framework can be applied to such problems as well.}

\section{Introduction}\label{sec:intro}
The nonoverlapping Robin--Robin method was introduced in~\cite{lions3} and shown to converge when applied to linear elliptic equations. Since then there has been several theoretical results concerning the method for both nonlinear elliptic and parabolic equations; see, e.g.,~\cite{Halpern10,gander23,quarteroni} and references therein.
    
Recently, we have derived an abstract approach to the convergence of domain decompositions methods, and in particular to the Robin--Robin method~\cite{EHEE22,EHEE24}. 
Our approach is based on the observation by~\cite{agoshkov83}; also see \cite{discacciati07}, that the Robin--Robin method can be reformulated into a Peaceman--Rachford iteration on the interface of the subdomains by making use of Steklov--Poincaré operators. More precisely, for two subdomains in space, or space-time, the Robin--Robin approximation can formally be written as $(u^n_1,u^n_2)=(F_1\eta^n,F_2\eta^n)$, where $F_i$ is a solution operator of the equation on a single subdomain, and 
the iterates $\eta^n$ on the interface of the subdomains are given by 
\begin{equation}\label{eq:pr}
    		 \eta^{n+1}=(sJ+S_2)^{-1}(sJ-S_1)(sJ+S_1)^{-1}(sJ-S_2)\eta^n,\quad n=0,1,2,\ldots
\end{equation}
Here, $S_i: Z\to Z^*$ denote the possibly nonlinear Steklov--Poincaré operators and $s>0$ is the method parameter. Moreover, $J:\mu\mapsto (\mu, \cdot)_H$ for some Gelfand triple $Z\hookrightarrow H\hookrightarrow Z^*$.

This reformulation in terms of Steklov--Poincaré operators means that both linear and nonlinear, elliptic and parabolic equations can all be treated within the same framework. For this abstract framework to be applicable their are only three requirements. First, the operators $S_1+S_2, sJ+S_i$ must be bijective. Second, the Steklov--Poincaré operators must satisfy a monotonicity property of the form  
\begin{equation}\label{eq:mon}
k(\|F_i\eta-F_i\mu\|_{X_i})\leq \langle S_i\eta-S_i\mu, \eta-\mu\rangle_{Z^*\times Z},
\end{equation}
where the function $k(x)>0$ tends to zero as $x$ tends to zero. Third, the solution $u$ to the original equation must have a normal derivative on the interface belonging to the Hilbert space $H$. 

By restricting the Steklov--Poincaré operators $S_i$ into maximal monotone operators $\mathcal{S}_i$ on $H$ and employing the assumed regularity of $u$, the abstract result~\cite{lionsmercier} yields the limit 
\begin{equation}\label{eq:limit}
\langle S_i\eta-S_i\eta^n, \eta-\eta^n\rangle_{Z^*\times Z}=(\mathcal{S}_i\eta-\mathcal{S}_i\eta^n, \eta-\eta^n)_H\to0\text{ as }n\to 0,
\end{equation}
where $\eta$ is the restriction of $u$ to the interface of the subdomains. See~\cite[Section~8]{EHEE22} for details. Combining this limit with~\cref{eq:mon} yields that the Robin--Robin approximation $(u^n_1,u^n_2)$ converges in the $X_1\times X_2$-norm.

The main two issues when studying nonlinear elliptic or parabolic equations, compared to linear elliptic problems, are that $Z$ might not be a Hilbert space and the weak formulation of the equation may require different test and trial spaces. The first issue arises when approximating nonlinear degenerate elliptic equations, where the bijectivity of the Steklov--Poincaré operators can be resolved by using the Browder--Minty theorem~\cite{EHEE22}. The aim of this short note is to illustrate the usage of the abstract framework in the context of the second issue. To this end, we consider linear parabolic equations with homogeneous initial, boundary data, i.e., 
\begin{equation}\label{eq:strong+}
 \left\{
        \begin{aligned}
                u_t-\nabla\cdot\bigl(\alpha(x)\nabla u\bigr)&=f  & &\text{in }\Omega\times\R^+,\\
               u&=0 & &\text{on }\partial\Omega\times\R^+\text{ and in }\Omega\times\{0\}.
            \end{aligned}
    \right.
\end{equation}
Here, the spatial Lipschitz domain $\Omega\subset\R^d$, $d=2,3$, is decomposed as
\begin{equation}\label{eq:domain}
\overline{\Omega}=\overline{\Omega}_1\cup\overline{\Omega}_2,\quad \Omega_1\cap\Omega_2=\emptyset,\quad\text{and}\quad\Gamma=(\partial\Omega_1\cap\partial\Omega_2)\setminus\partial\Omega. 
\end{equation}
Even this simple setting gives rise to a weak formulation with different test and trial spaces. Furthermore, the standard parabolic setting with the trial space in $H^1\bigl(\R^+,H^{-1}(\Omega)\bigr)$ does not give rise to a well defined transmission problem. Instead we will employ a $H^{1/2}$-setting for the temporal regularity and prove the bijectivity of the Steklov--Poincaré operators via the Banach--Ne\v{c}as--Babu\v{s}ka theorem. Convergence is then obtained in $X_i=L^2\bigl(\R^+,H^1(\Omega_i)\bigr)$. Unlike previous studies, the Robin--Robin naturally preserves the homogeneous initial condition in this setting and no further regularity assumptions are required regarding the numerical iterates or the subdomain's boundaries. 
\section{Preliminaries}\label{sec:prel}
  We will assume that the following holds, which is a requirement for defining the trace operator and the Sobolev spaces on $\partial\Omega_i$ and $\Gamma$.
 \begin{assumption}\label{ass:dom}
        The subdomains $\Omega_i$ are Lipschitz and bounded. The interface $\Gamma$ and exterior boundaries $\partial\Omega\setminus\partial\Omega_{i}$ are $(d-1)$-dimensional Lipschitz manifolds.
    \end{assumption}
    The spaces on the spatial domains are defined as
    \begin{displaymath}
        V=H_0^1(\Omega),\quad  V_i^0=H_0^1(\Omega_i),\quad\text{and}\quad
        V_i=\{v\in H^1(\Omega_i): \restr{(\hat{T}_{\partial\Omega_i}v)}{\partial\Omega_i\setminus\Gamma}=0\},
    \end{displaymath}
    where $\hat{T}_{\partial\Omega_i}:H^1(\Omega_i)\rightarrow H^{1/2}(\partial\Omega_i)$ is the trace operator, see~\cite[Theorem 6.8.13]{kufner} for details. Moreover, we define the fractional Sobolev space $H^{1/2}(\partial\Omega_i)$ as in~\cite[p. 591]{EHEE22}. The spatial Lions--Magenes space, see e.g.~\cite{EHEE24}, is denoted by $\Lambda$. We define the spatial interface trace operator $\hat{T}_i:V_i\rightarrow \Lambda:u\mapsto\restr{(\hat{T}_{\partial\Omega_i}u)}{\Gamma}$ and note that this is a bounded linear operator, see e.g.~\cite[Lemma 4.4]{EHEE22}.
    
    For the temporal fractional Sobolev spaces $H^s(\R)$ we use the Fourier characterization, see~\cite[(3.2)]{EHEE24} for a full definition. Next, we define the fractional Sobolev spaces on $\R^+$ and the temporal Lions--Magenes space by
\begin{equation*}
     \begin{aligned}
         H^s(\R^+)&=\{u\in L^2(\R^+): \hat{E}_\text{even} u\in H^s(\R)\}& &\text{with} &\|u\|_{H^s(\R^+)} &=\|\hat{E}_\text{even} u\|_{H^s(\R)},\\
    	    H_{00}^{1/2}(\R^+)&=\{u\in L^2(\R^+): \hat{E}_\R u\in H^{1/2}(\R)\} & &\text{with} &\|u\|_{H_{00}^{1/2}(\R^+)} &=\|\hat{E}_\R u\|_{H^{1/2}(\R)},
     \end{aligned}
\end{equation*}
    respectively. Here $\hat{E}_\R$ is the extension by zero and $\hat{E}_\text{even}$ is the even extension.
    We will make use of the Bochner--Sobolev spaces $L^2(\R^+, Y)$, $H^s(\R^+, Y)$, and $H_{00}^{1/2}(\R^+, Y)$, where $Y$ denotes an abstract Hilbert space. According to~\cite[Lemma 2]{EHEE24} our spatial operators $\hat{T}_{\partial\Omega_i}$, $\hat{T}_i$, $\hat{E}_\R$, $\hat{E}_{\text{even}}$ can be extended as follows
    \begin{align*}
        &T_{\partial\Omega_i}:L^2\bigl(\R^+,  H^1(\Omega_i)\bigr) \rightarrow L^2\bigl(\R^+,  H^{1/2}(\partial\Omega_i)\bigr),\,
        T_i:L^2(\R^+,  V_i)\rightarrow L^2(\R^+,  \Lambda),\\
        &E_\R: L^2\bigl(\R^+,  L^2(\Omega_i)\bigr)\rightarrow L^2\bigl(\R,  L^2(\Omega_i)\bigr),\, E_\text{even}: L^2\bigl(\R^+,  L^2(\Omega_i)\bigr)\rightarrow L^2\bigl(\R,  L^2(\Omega_i)\bigr).
    \end{align*}
   Moreover, we have the relations
    \begin{align}
                H^s\bigl(\R^+, L^2(\Omega_i)\bigr)&=\{u\in L^2\bigl(\R^+,  L^2(\Omega_i)\bigr): E_{\text{even}} u\in H^s\bigl(\R, L^2(\Omega_i)\bigr)\},\label{eq:H12}\\
                H^{1/2}_{00}\bigl(\R^+, L^2(\Omega_i)\bigr)&=\{u\in L^2\bigl(\R^+,  L^2(\Omega)\bigr): E_\R u\in H^{1/2}\bigl(\R, L^2(\Omega_i)\bigr)\},\label{eq:H1200}
    \end{align}
    with equivalent norms
    \begin{align*}
                \|u\|_{H^s(\R^+, L^2(\Omega_i))}=\|E_{\text{even}}u\|_{H^s(\R, L^2(\Omega_i))},\,\|u\|_{H^{1/2}_{00}(\R^+, L^2(\Omega_i))}=\|E_\R u\|_{H^{1/2}(\R, L^2(\Omega_i))}.
    \end{align*}
    
    We introduce the Hilbert spaces
\begin{equation*}
     \begin{aligned}
            W &= H^{1/2}_{00}\bigl(\R^+, L^2(\Omega)\bigr)\cap L^2\bigl(\R^+, V\bigr),  &\Tilde{W} &= H^{1/2}\bigl(\R^+, L^2(\Omega)\bigr)\cap L^2\bigl(\R^+, V\bigr),\\
            W_i &= H^{1/2}_{00}\bigl(\R^+, L^2(\Omega_i)\bigr)\cap L^2\bigl(\R^+, V_i\bigr), &\Tilde{W}_i &= H^{1/2}\bigl(\R^+, L^2(\Omega_i)\bigr)\cap L^2\bigl(\R^+, V_i\bigr),\\
        W_i^0 &= H^{1/2}_{00}\bigl(\R^+, L^2(\Omega_i)\bigr)\cap L^2\bigl(\R^+, V_i^0\bigr), &\Tilde{W}_i^0 &= H^{1/2}\bigl(\R^+, L^2(\Omega_i)\bigr)\cap L^2\bigl(\R^+, V_i^0\bigr),\\
        Z&=H^{1/4}\bigl(\R^+, L^2(\Gamma)\bigr)\cap L^2\bigl(\R^+, \Lambda\bigr). & & &
        \end{aligned}
\end{equation*}
    Finally, we define the sets $\D=C^\infty_0\bigl(\R^+, C^\infty_0(\Omega)\bigr)$ and $\D_i=C^\infty_0\bigl(\R^+, C^\infty(\overline{\Omega_i})\bigr)$.
    \begin{lemma}\label{lemma:dense}
        The set $\D$ is dense in $H^{1/2}\bigl(\R^+, L^2(\Omega)\bigr)$, $H^{1/2}_{00}\bigl(\R^+, L^2(\Omega)\bigr)$, and $L^2\bigl(\R^+, V\bigr)$. The set $\D_i$ is dense in $H^{1/2}\bigl(\R^+, L^2(\Omega_i)\bigr)$, $H^{1/2}_{00}\bigl(\R^+, L^2(\Omega_i)\bigr)$, and $L^2\bigl(\R^+, H^1(\Omega_i)\bigr)$.
    \end{lemma}
    \begin{proof}
        We first recall that $C^\infty_0(\R^+)$ is dense in $L^2(\R^+)$ and $H^{1/2}(\R^+)$; see~\cite[Theorem 11.1]{lionsmagenes1}. By the interpolation identity $H^{1/2}_{00}=[H^1_0(\R^+), L^2(\R^+)]_{1/2}$; see~\cite[Theorem 11.7, Remark 2.6]{lionsmagenes1}, we also have that $H^1_0(\R^+)$ is dense in $H_{00}^{1/2}(\R^+)$, which, by definition of $H^1_0(\R^+)$ and~\cite[Proposition 2.3]{lionsmagenes1}, implies that $C^\infty_0(\R^+)$ is dense in $H_{00}^{1/2}(\R^+)$. Moreover, $C^\infty_0(\Omega)$ is dense in $L^2(\Omega)$ and $V$; see~\cite[Theorem 2.6.1]{kufner}. Finally, recall that $C^\infty(\overline{\Omega_i})$ is dense in $L^2(\Omega_i)$ and $H^1(\Omega_i)$; see~\cite[Theorem 2.6.1, Theorem 5.5.9]{kufner}. The result now follows from~\cite[Theorem 3.12]{weidmann}.
    \end{proof}
    The trace operator defined on $W_i$ has the following behaviour. The statement follows from~\cite[Lemma 2.4, Corollary 2.12]{costabel90} using the same techniques as in~\cite[Lemma 5]{EHEE24}.
    \begin{lemma}\label{lemma:Ti}
        The trace operator is bounded as an operator $T_i: W_i\rightarrow Z$ and $T_i: \tilde{W}_i\rightarrow Z$. Moreover, there exists a bounded right inverse $R_i:Z\rightarrow W_i$.
    \end{lemma}
    \begin{remark}
        The equation requires different trial and test spaces, $W_i$ and $\tilde{W}_i$, respectively. However, due to the fact that they share the same trace space $Z$ the Steklov--Poincaré theory can be formulated using only one space $Z$. Moreover, the inclusion $W_i\hookrightarrow \tilde{W}_i$ means that the extension operator is also bounded as $R_i:Z\rightarrow \tilde{W}_i$, which is required for the Steklov--Poincaré operators to be well defined.
    \end{remark}
    
    \section{Weak formulations of parabolic equations}\label{sec:weak}
    To perform our analysis we make the following assumption on the equation~\cref{eq:strong+}.
    \begin{assumption}\label{ass:eq}
        The equation~\cref{eq:strong+} satisfies the following.
        \begin{itemize}
            \item The function $\alpha\in L^\infty(\Omega)$ satisfies the bound $\alpha(x)\geq c>0$ for a.e.\ $x\in\Omega$. 
                        \item We have $f\in \tilde{W}^*$ and there exist $f_i\in \Tilde{W}_i^*$ such that
            \begin{equation*}
                \langle f, v\rangle =\langle f_1, \restr{v}{\Omega_1\times\R^+}\rangle+\langle f_2, \restr{v}{\Omega_2\times\R^+}\rangle\quad \text{for all }v\in \tilde{W}.
            \end{equation*}
        \end{itemize}
    \end{assumption}
    We introduce the operator $A_i:W_i\rightarrow \Tilde{W}_i^*$ as the extension of
    \begin{align*}
        \langle A_iu_i, v_i\rangle=\int_{\R^+}\int_{\Omega_i}\partial_tu_i\,v_i+\alpha(x)\nabla u_i\cdot\nabla v_i\,\,\mathrm{d}x\,\mathrm{d}t,
    \end{align*}
     where $u_i, v_i\in\D_i$. The operator $A:W\rightarrow \tilde{W}^*$ is defined similarly using $u, v\in\D$.
    \begin{lemma}\label{lemma:ai}
        Suppose that~\cref{ass:dom,ass:eq} hold. The operators $A_i:W_i\rightarrow \Tilde{W}_i^*$ and $A:W\rightarrow \Tilde{W}^*$ are bounded linear operators and one has the $A_i$-bounds
        \begin{equation}\label{eq:Aico}
            \langle A_iu, u\rangle\geq c\|u\|_{L^2(\R^+, V_i)}^2\quad\text{for all }u\in W_i.
        \end{equation}
        Moreover, there exists a bounded linear operator $B_i:W_i\rightarrow\tilde{W}_i$ such that
        \begin{equation*}
            \langle A_iu, B_iu\rangle\geq c\|u\|_{W_i}^2\quad\text{for all }u\in W_i.
        \end{equation*}
    \end{lemma}
    \begin{proof}
    We prove the statement for $A_i$ since the case of $A$ follows similarly. We write $A_i=A_i^t+A_i^s$, where
       \begin{equation*}
            \langle A_i^tu, v\rangle=\int_{\R^+}\int_{\Omega_i}\partial_tu\,v\dx\dt\quad\text{and}\quad\langle A_i^su, v\rangle=\int_{\R^+}\int_{\Omega_i}\alpha(x)\nabla u\cdot\nabla v\dx\dt.
       \end{equation*}
        We consider first the temporal term. The identities~\cref{eq:H12,eq:H1200} then yield
        \begin{align*}
            &|\langle A_i^tu, v\rangle|=\biggl|\int_{\R^+}\int_{\Omega_i}\partial_tuv\dx\dt\biggr| =\biggl|\int_\R\int_{\Omega_i}\partial_tE_\R u \,E_{\text{even}}v\dx\dt\biggr|\\
            &\quad\leq C\|E_\R u\|_{H^{1/2}(\R, L^2(\Omega_i))}\|E_{\text{even}}v\|_{H^{1/2}(\R, L^2(\Omega_i))}\\
            &\quad\leq C\|u\|_{H_{00}^{1/2}(\R^+, L^2(\Omega_i))}\|v\|_{H^{1/2}(\R^+, L^2(\Omega_i))},
        \end{align*}
        where the first inequality follows as in~\cite[Section 5]{EHEE24}. This together with~\cref{lemma:dense} shows that $A_i^t$ extends to a bounded linear operator $A_i^t:H_{00}^{1/2}\bigl(\R^+, L^2(\Omega_i)\bigr)\rightarrow H^{1/2}\bigl(\R^+, L^2(\Omega_i)\bigr)^*$. Using this continuity, it is easy to verify that $\langle A_i^tu, u\rangle=0$ for all $u\in H_{00}^{1/2}\bigl(\R^+, L^2(\Omega_i)\bigr)$. Next, we define $B_i^\varphi=R_{\R^+}(\cos(\varphi)I-\sin(\varphi)\Hc_i)E_\R$ for $\varphi\in(0,\pi/2)$. Here, $R_{\R^+}:H^{1/2}\bigl(\R, L^2(\Omega_i)\bigr)\cap L^2(\R, V_i)\rightarrow \tilde{W}_i$ denotes the bounded linear operator given by the restriction to $\R^+$ and
        \begin{equation*}
            \Hc_i:H^{1/2}\bigl(\R, L^2(\Omega_i)\bigr)\cap L^2(\R, V_i)\rightarrow H^{1/2}\bigl(\R, L^2(\Omega_i)\bigr)\cap L^2(\R, V_i)
        \end{equation*}
        denotes the Hilbert transform; see~\cite[Section 4]{EHEE24}. By~\cite[(5.5)]{EHEE24} we have
        \begin{equation}\label{eq:Aitcoer2}
        \begin{aligned}
            \langle A_i^tu, B_i^\varphi u\rangle&=\int_\R\int_{\Omega_i}\partial_tE_\R u\,\bigl(\cos(\varphi)I-\sin(\varphi)\Hc_i\bigr)E_\R u\dx\dt\\
            &=\sin(\varphi)\|E_\R u\|_{H^{1/2}(\R^+, L^2(\Omega_i))}^2=\sin(\varphi)\|u\|_{H_{00}^{1/2}(\R^+, L^2(\Omega_i))}^2 
        \end{aligned}
        \end{equation}
        for all $u\in\D_i$. By continuity~\cref{eq:Aitcoer2} also holds for all $u\in H_{00}^{1/2}(\R^+, L^2(\Omega_i))$. 
        We now consider the spatial term. A standard argument shows that $A_i^s: L^2(\R^+, V_i)\rightarrow L^2(\R^+, V_i^*)\cong L^2(\R^+, V_i)^*$ is bounded and coercive. Finally, the fact that $\Hc_i$ is bounded yields 
        \begin{align*}
            &\langle A_i^su, B_i^\varphi u\rangle\geq \bigl(c\cos(\varphi)-C\sin(\varphi)\bigr)\|u\|_{L^2(\R^+, V_i)}^2\quad\text{for all }u\in L^2(\R^+, V_i).
        \end{align*}
        Putting this together, the operator $A_i=A_i^s+A_i^t$ extends to a continuous linear operator $A_i:W_i\rightarrow \tilde{W}_i^*$ that satisfies the bounds~\cref{eq:Aico} and
        \begin{align*}
            &\langle A_iu, B_i^\varphi u\rangle\geq c\sin(\varphi)\|u\|_{H_{00}^{1/2}(\R^+, L^2(\Omega_i))}^2+\bigl(c\cos(\varphi)-C\sin(\varphi)\bigr)\|u\|_{L^2(\R^+, V_i)}^2
        \end{align*}
        for all $u\in W_i$. Choosing $B_i=B_i^\varphi$ for $\varphi>0$ small enough finishes the proof.
    \end{proof}
    The weak formulation of the equation~\cref{eq:strong+} is to find $u\in W$ such that
    \begin{equation}\label{eq:weak}
	    \langle Au, v\rangle=\langle f, v\rangle\quad \textrm{for all }v\in \tilde{W}.
    \end{equation}
    Under~\cref{ass:eq} the weak problem has a unique solution; see~\cite[Corollary 3.9]{schwab17}. We also need the following existence result for solutions to the problems on $\Omega_i\times \R^+$ with nonhomogenous boundary data.
 \begin{lemma}\label{lemma:Fi}
Suppose that~\cref{ass:dom,ass:eq} hold. For $g\in (\Tilde{W}_i^0)^*$ and $\eta\in Z$ there exists a unique $u\in W_i$ such that $T_iu=\eta$ and
\begin{equation}\label{eq:nonhomo}
            \langle A_iu, v\rangle=\langle g, v\rangle\quad \textrm{for all }v\in \tilde{W}_i^0.
\end{equation}
The solution $u$ also satisfies the bound $\|u\|_{W_i}\leq C\bigl(\|g\|_{(\Tilde{W}_i^0)^*}+\|\eta\|_Z\bigr)$.
 \end{lemma}
The proof follows by first applying~\cite[Corollary 3.9]{schwab17} to 
 \begin{equation*}
    	    \langle A_iu_0, v\rangle=\langle g-KA_iR_i\eta, v\rangle\quad \textrm{for all }v\in \tilde{W}_i^0,
\end{equation*}
for which the unique solution $u_0$ satisfies the bound $\|u_0\|_{W_i}\leq C\|g-KA_iR_i\eta\|_{(\tilde{W}_i^0)^*}$. Using~\cite[(4.2)]{EHEE24} for $\R^+$ yields $T_iu_0=0$, which shows that $u=u_0+R_i\eta$ is the unique solution to~\cref{eq:nonhomo}, and the desired bound follows by the corresponding bound for $u_0$ together with~\cref{lemma:Ti,lemma:ai}. Here $K: \tilde{W}_i^*\rightarrow(\tilde{W}_i^0)^*:\ell\mapsto \restr{\ell}{\tilde{W}_i^0}$ denotes the bounded and linear, but not necessarily injective, inclusion map.

Applying~\cref{lemma:Fi} with $g=0$ or $\eta=0$ yields the bounded solution operators $F_i:Z\rightarrow W_i$ and $G_i:(\Tilde{W}_i^0)^*\rightarrow W_i^0$, respectively.
   
    \section{Transmission problem and Steklov--Poincaré operators}\label{sec:tran}
  The transmission problem is to find $(u_1, u_2)\in W_1\times W_2$ such that
    \begin{equation}\label{eq:weaktran}
    	\left\{\begin{aligned}
    	     \langle A_i u_i, v_i\rangle &=\langle f_i, v_i \rangle & & \text{for all } v_i\in \tilde{W}_i^0,\, i=1,2,\\
    	     T_1u_1&=T_2u_2, & &\\
    	     \textstyle\sum_{i=1}^2 \langle A_i  u_i, R_i\mu\rangle-\langle f_i, R_i\mu\rangle &=0 & &\text{for all }\mu\in Z. 
    	\end{aligned}\right.
    \end{equation}
    Before discussing the equivalence of the weak equation and the transmission problem, we need to be able to glue together functions in our Hilbert spaces without loss of regularity. The result follows by a tensor basis argument, see~\cite[Lemma 9]{EHEE24}.
    \begin{lemma}\label{lemma:paste}
Suppose that~\cref{ass:dom} holds. If $u\in W$ then $u_i=\restr{u}{\Omega_i\times\R^+}$ satisfy $u_i\in W_i$ and $T_1u_1=T_2u_2$. Conversely, if  $u_i\in W_i$ and $T_1u_1=T_2u_2$ then $u=\{u_i \textrm{ on } \Omega_i\times\R^+\, , i=1,2\}$ satisfies $u\in W$. The same result holds with $(W,W_i)$ replaced by $(\tilde{W},\tilde{W}_i)$.
    \end{lemma}
 After establishing~\cref{lemma:paste} the equivalence of the weak equation and the transmission problem now follows in the same way as for linear elliptic equations~\cite[Lemma 1.2.1]{quarteroni}; also see~\cite[Remark 2]{EHEE24}.
    \begin{lemma}
        Suppose that~\cref{ass:dom} holds. If $u$ solves~\cref{eq:weak} then $(u_1, u_2)=(\restr{u}{\Omega_1\times\R^+},\restr{u}{\Omega_2\times\R^+})$ solves~\cref{eq:weaktran}. Conversely, if $(u_1, u_2)$ solves~\cref{eq:weaktran} then $u=\{u_i \textrm{ on } \Omega_i\times\R^+\, , i=1,2\}$ solves~\cref{eq:weak}.
    \end{lemma}

The Steklov--Poincaré operators and interface source terms are defined as
    \begin{align*}
    \langle S_i\eta, \mu\rangle=\langle A_iF_i\eta, R_i\mu\rangle\quad\text{and}\quad\langle\chi_i,\mu\rangle=\langle f_i-A_iG_if_i, R_i\mu\rangle,
    \end{align*}
    respectively. The transmission problem can now be reformulated as the Steklov--Poincaré equation by setting $\eta=T_iu_i$ and $u_i=F_i\eta+G_if_i$. This gives that the transmission problem is equivalent to finding $\eta\in Z$ such that
    \begin{equation}\label{eq:speq}
    	   \textstyle\sum_{i=1}^2 \langle S_i\eta, \mu\rangle=\sum_{i=1}^2 \langle \chi_i, \mu\rangle\quad \textrm{for all } \mu\in Z.
    \end{equation}
    This follows by simply considering the definition of the Steklov--Poincaré operators. We can now validate the bijectivity and the monotonicity properties stated in the introduction.
\begin{theorem}
        Suppose that~\cref{ass:dom,ass:eq} hold. The operators $S_i:Z\rightarrow Z^*$ are then bounded and fulfill the monotonicity property~\cref{eq:mon} with $(k,X_i)=\bigl((\cdot)^2,L^2(\R^+, V_i)\bigr)$. Furthermore, the operators $S_1+S_2$ and $sJ+S_i$ are bijective.
    \end{theorem}
    \begin{proof}
        The boundedness and coercivity follow directly as the case for $\R$, see e.g.~\cite[Lemma 13]{EHEE24}. To prove bijectivity we employ the Banach--Ne\v{c}as--Babu\v{s}ka theorem, see e.g.~\cite[Theorem 2.6]{ern}. For simplicity we first prove that $S_i$ is bijective. We define $P=T_iB_iF_i$, where $B_i$ is as in~\cref{lemma:ai} and $P$ is independent of $i$ due to the commutative property in~\cite[Lemma 7]{EHEE24}. The inf-sup condition follows from
        \begin{align*}
            &\langle S_i\eta, P\eta\rangle=\langle A_iF_i\eta, F_i T_iB_iF_i\eta\rangle=\langle A_iF_i\eta, B_iF_i\eta\rangle\\
            &\qquad+\langle A_iF_i\eta, (F_i T_iB_iF_i-B_iF_i)\eta\rangle=\langle A_iF_i\eta, B_iF_i\eta\rangle\geq \|F_i\eta\|_{W_i}^2\geq c\|\eta\|_Z^2,
        \end{align*}
        where we have used that $(F_i T_iB_iF_i-B_iF_i)\eta\in \tilde{W}_i^0$. The adjoint injectivity condition is
        \begin{align*}
            \langle S_i\mu, \mu\rangle\geq c\|\mu\|_{L^2(\R^+, \Lambda)}^2>0\quad\text{for }\mu\neq 0.
        \end{align*}
        The proof for $S_1+S_2$ is similar and the proof for $sJ+S_i$ follows by the same argument and the facts that $\langle J\eta, P\eta\rangle\geq 0$ and $\langle J\eta, \eta\rangle\geq 0$.
    \end{proof}
    For the convergence of the Robin--Robin method we require the following assumption.
    \begin{assumption}\label{ass:reg}
        Let $u$ be the solution to~\cref{eq:weak}. The linear functionals
        \begin{displaymath}
        \mu\mapsto \langle A_i\restr{u}{\Omega_i\times\R^+}, R_i\mu\rangle-\langle f_i, R_i\mu\rangle,\quad{i=1,2},
        \end{displaymath}
        are in $H^*=L^2(\Gamma\times\R^+)^*$.
    \end{assumption}
The convergence now follows from~\cite[Proposition~1]{lionsmercier}, as described in~\cref{sec:intro}.
\begin{theorem}
If~\cref{ass:dom,ass:eq,ass:reg} hold, then the iterates $(u_1^n, u_2^n)$ of the Robin--Robin method converges to the solution $(u_1, u_2)$ of~\cref{eq:weaktran} in $L^2(\R^+, V_1)\times L^2(\R^+, V_2)$.
\end{theorem}
\bibliographystyle{plain}
\bibliography{ref}

\begin{acknowledgement}
Funding: This work was supported by the Swedish Research Council under the grants 2019--05396 and 2023--04862.
\end{acknowledgement}

\end{document}